\documentclass{amsart}
\usepackage{amsfonts}
\usepackage{amsmath}
\usepackage{amssymb}
\usepackage{amsthm}
\usepackage{graphicx}
\usepackage{capt-of}
\usepackage{tikz}
\usetikzlibrary{patterns}
\usetikzlibrary{decorations.markings, patterns}
\usepackage{multirow}
\usepackage{array} 
\newcolumntype{C}{>{$}c<{$}}

\newtheorem{theorem}{Theorem}
\newtheorem{prop}{Proposition}
\newtheorem{lemma}{Lemma}
\newtheorem{rem}{Remark}
\newtheorem{exmp}{Example}

\begin{document}
\title[Zigzags in combinatorial tetrahedral chains]{Zigzags in combinatorial tetrahedral chains and the associated Markov chain}
\author{Adam Tyc}
\subjclass[2000]{}

\
\address{Faculty of Mathematics and Computer Science, 
University of Warmia and Mazury, S{\l}oneczna 54, Olsztyn, Poland}
\email{adamtyc.math@gmail.com}

\maketitle
\begin{abstract}
Zigzags in graphs embedded in surfaces are cyclic sequences of edges whose any two consecutive edges are different, have a common vertex and belong to the same face. 
We investigate zigzags in randomly constructed combinatorial tetrahedral chains. 
Every such chain contains at most $3$ zigzags up to reversing. 
The main result is the limit of the probability that a randomly constructed tetrahedral chain contains precisely $k\in\{1,2,3\}$ zigzags up to reversing as its length approaches infinity. 
Our key tool is the Markov chain whose states are types of $z$-monodromies. 
\end{abstract}

\section{Introduction}
A tetrahedral chain is a sequence of regular tetrahedra where any two consecutive tetrahedra are glued together face to face. 
In the {\it New Scottish Book} (Problem 288) and in \cite{HS-prob} Hugo Steinhaus conjectured that there is no closed tetrahedral chain. 
This fact is proved by Stanisław Świerczkowski \cite{SSw}. 
Michael Elgersma and Stan Wagon shown that there is a tetrahedral chain with arbitrarily small discrepancy from closure \cite{ElWa}. 
Some results on tetrahedral chains can be found in \cite{BJ, Stewart}. 


We say that a tetrahedral chain is {\it proper} if it does not self-intersect. 
For some proper tetrahedral chains there are faces that cannot be used to gluing the next tetrahedron. 
Thus, faces in such tetrahedral chains cannot be chosen completely randomly. 
For this reason, we do not assume that tetrahedra are regular and consider {\it combinatorial tetrahedral chains}, i.e. sequences of graphs $K_4$ embedded in spheres and glued together face to face. 
In the present paper, we investigate {\it zigzags} in combinatorial tetrahedral chains. 

Recall that zigzags in graphs embedded in surfaces are sequences of edges such that any two consecutive edges are different, have a common vertex and belong to the same face \cite{DDS-book}. 
Zigzags in regular polyhedra are skew polygons without self-intersections. 
For this reason, Coxeter called them {\it Petrie polygons} in \cite{Coxeter}. 
Zigzags are known also as {\it closed left-right paths} \cite{GR-book, Shank}. 
Zigzags are used in mathematical chemistry to enumerating all combinatorial possibilities for fullerenes \cite{BD, DDS-book}
and they are exploited in computer graphics \cite{TriApp}. 
{\it $Z$-knotted} embedded graphs (embedded graphs containing a single zigzag) are related to {\it Gauss code problem}: 
the zigzag of $z$-knotted graph is a certain Gauss code; in other words, $z$-knotted graphs are geometrical realizations of Gauss codes \cite{CrRos, GR-book, Lins2}. 
More results concerning zigzags can be found in \cite{PT1, PT3, PT2, T2, T1}.

A tetrahedral chain contains at most $3$ zigzags (up to reversing). 
Our main result concerns the probability that a randomly step-by-step constructed tetrahedral chain contains precisely $k\in\{1,2,3\}$ zigzags (up to reversing). 
We determine the limit of this probability as the length of combinatorial tetrahedral chain approaches infinity. 
Our reasonings are as follows. 
The fact that a combinatorial tetrahedral chain contains precisely $k$ zigzags (up to reversing) will be reformulated in terms of $z$-monodromies (Section 5). 
Next, we construct a digraph whose vertices are all possible types of $z$-monodromies (Section 6). 
Recall that there are precisely $7$ types of $z$-monodromies which will be denoted by (M1)--(M7) \cite{PT2}. 
We determine types of $z$-monodromies of faces which arise by partition of a face with $z$-monodromy (M$i$). 
This operation on a face is a step in the construction of a combinatorial tetrahedral chain. 
The digraph contains an edge from (M$i$) to (M$j$) if by gluing a tetrahedron to a face with the $z$-monodromy (M$i$) we obtain the $z$-monodromy (M$j$) for at least one of the three new faces. 
For each edge from (M$i$) to (M$j$) we assign a probability that in the construction after gluing a tetrahedron to a face with the $z$-monodromy (M$i$) we choose a face with the $z$-monodromy (M$j$). 
This digraph is a diagram of an ergodic Markov chain and the main result follows from its well-known properties (Section 7). 


\section{Zigzags in triangulations}

Let $\Gamma$ be a {\it triangulation} of connected closed $2$-dimensional surface $M$ (not necessarily orientable), i.e. 
a $2$-cell embedding of a connected simple finite graph in $M$ whose all faces are triangles \cite[Section 3.1]{MT-book}. 
Then $\Gamma$ satisfies the following two conditions: 
(1) every edge is contained in precisely two distinct faces, 
(2) the intersection of two distinct faces is an edge or a vertex or the empty set. 
Two distinct edges are said to be {\it adjacent} if there is a face containing both of them. 
Since each face of $\Gamma$ is a triangle, any two adjacent edges have a common vertex. 
Two faces are {\it adjacent} if their intersection is an edge. 

A {\it zigzag}  in $\Gamma$ is a sequence of edges $\{e_{i}\}_{i\in {\mathbb N}}$ satisfying the following conditions for every $i\in {\mathbb N}$:
\begin{enumerate}
\item[$\bullet$] $e_{i},e_{i+1}$ are adjacent, 
\item[$\bullet$] the faces containing $e_{i},e_{i+1}$ and $e_{i+1},e_{i+2}$ are distinct and the edges $e_{i}$ and $e_{i+2}$ are disjoint. 
\end{enumerate} 
Since $\Gamma$ is finite, 
for every zigzag $Z=\{e_{i}\}_{i\in {\mathbb N}}$
there is a natural number $n\geq 1$ such that $e_{i+n}=e_{i}$ for every $i\in {\mathbb N}$. 
Thus, $Z$ can be considered as a cyclic sequence
$e_{1},\dots, e_{n}$, where $n$ is the smallest number satisfying this condition. 
Observe that $Z$ can be written as a cyclic sequence of vertices $v_1,\dots,v_n$, where $v_i$ and $v_{i+1}$ form an edge $e_i$ for $i=1,\dots, n-1$ and the edge $e_n$ consists of $v_n$ and $v_1$. 

Each ordered pair of edges contained in a face completely determines a zigzag. 
Conversely, any zigzag is completely determined by any pair of consecutive edges belonging to it. 
If $Z=\{e_{1},\dots, e_{n}\}$ is a zigzag, then the sequence $Z^{-1}=\{e_{n},\dots, e_{1}\}$ also is a zigzag and we call it {\it reversed} to $Z$. 
Each zigzag $Z$ is not self-reversed, i.e. $Z\neq Z^{-1}$ (see, for example, \cite{PT2}). 
Denote by $\mathcal{Z}(\Gamma)$ the set of all zigzags in $\Gamma$. 
Thus, $|\mathcal{Z}(\Gamma)|$ is even. 
If $|\mathcal{Z}(\Gamma)|=2k$, then we say that $\Gamma$ contains precisely $k$ zigzags {\it up to reversing}. 

If all edges in $Z$ are mutually distinct, then $Z$ is {\it edge-simple} (note that vertices can repeat in an edge-simple zigzag). 
The triangulation $\Gamma$ is {\it $z$-knotted} if $|\mathcal{Z}(\Gamma)|=2$, i.e. $\Gamma$ contains a single zigzag up to reversing. 
Let $F$ be a face in $\Gamma$ and let $\mathcal{Z}(F)$ be the set of all zigzags containing edges of $F$. 
We say that $\Gamma$ is {\it locally $z$-knotted} for $F$ if $|\mathcal{Z}(F)|=2$; in other words, there is a single zigzag (up to reversing) which contains edges of $F$. 
A triangulation is $z$-knotted if and only if it is locally $z$-knotted for each face \cite[Theorems 4.4 and 4.7]{PT2}. 

\begin{exmp}\label{ex1}{\rm 
Consider a tetrahedron $T$, i.e. an embedding of the graph $K_4$ in a sphere $\mathbb{S}^2$. 
Denote vertices of $T$ by $v_1,v_2,v_3,v_4$. 
Let $e_{ij}$ be the edge $v_iv_j$ (our edges are non-oriented, so $e_{ij}=e_{ji}$) and let $F_i$ be the face of $T$ which does not contain $v_i$ (where $i,j\in\{1,2,3,4\})$. 
\begin{center}
\begin{tikzpicture}[scale=0.4]


\coordinate (A) at (0.5,4);
\coordinate (1) at (0,-1);
\coordinate (3) at (-2,0.25);
\coordinate (2) at (3,0.25);

\draw[fill=black] (A) circle (3.5pt);

\draw[fill=black] (1) circle (3.5pt);
\draw[fill=black] (3) circle (3.5pt);
\draw[fill=black] (2) circle (3.5pt);

\draw[thick] (2)--(1)--(3);
\draw[thick, dashed] (3)--(2);

\draw[thick] (A)--(1);
\draw[thick] (A)--(2);
\draw[thick] (A)--(3);

\node at (0.5,4.5) {$v_4$};

\node at (0,-1.6) {$v_1$};
\node at (-2.65,0.25) {$v_3$};
\node at (3.65,0.25) {$v_2$};

\end{tikzpicture}
\captionof{figure}{ }
\end{center}
The set $\mathcal{Z}(T)$ consists of precisely $6$ zigzags:
$$e_{12},e_{23},e_{34},e_{41};\hspace{1cm}e_{12},e_{24},e_{43},e_{31};\hspace{1cm}e_{14},e_{42},e_{23},e_{31}$$
and their reverses. 
All these zigzags are edge-simple. 
Observe that each zigzag passes through all faces of $T$, i.e. $\mathcal{Z}(T)=\mathcal{Z}(F_i)$ for $i=1,2,3,4$. 
}\end{exmp}

\begin{exmp}\label{ex2}{\rm 
A $3$-gonal bipyramid $BP_3$ consists of a $3$-cycle embedded in $\mathbb{S}^2$ whose vertices are connected with two disjoint vertices (see Fig. 2). 
Denote the vertices of $3$-cycle by $1,2,3$ and the two other vertices by $a,b$. 
\begin{center}
\begin{tikzpicture}[scale=0.3]


\coordinate (A) at (0.5,4);
\coordinate (B) at (0.5,-5);
\coordinate (1) at (0,-1);
\coordinate (3) at (-2,0);
\coordinate (2) at (3,0);

\draw[fill=black] (A) circle (3.5pt);
\draw[fill=black] (B) circle (3.5pt);

\draw[fill=black] (1) circle (3.5pt);
\draw[fill=black] (3) circle (3.5pt);
\draw[fill=black] (2) circle (3.5pt);

\draw[thick] (2)--(1)--(3);
\draw[thick, dashed] (3)--(2);

\draw[thick] (A)--(1);
\draw[thick] (A)--(2);
\draw[thick] (A)--(3);

\draw[thick] (B)--(1);
\draw[thick] (B)--(2);
\draw[thick] (B)--(3);

\node at (0.5,4.6) {$a$};
\node at (0.5,-5.8) {$b$};

\node at (0.5,-1.5) {$1$};
\node at (-2.6,0) {$3$};
\node at (3.6,0) {$2$};

\end{tikzpicture}
\captionof{figure}{ }
\end{center}
This triangulation contains a single zigzag up to reversing 
$$a1,12,2b,b3,31,1a,a2,23,3b,b1,12,2a,a3,31,1b,b2,23,3a.$$
Thus, $BP_3$ is $z$-knotted. 
For this reason, both zigzags pass through all faces of $BP_3$. 
}\end{exmp}

\section{Combinatorial tetrahedral chains}

Let $\Gamma$ and $\Gamma'$ be triangulations of connected closed $2$-dimensional surfaces $M$ and $M'$ (respectively). 
Suppose that $F$ is a face in $\Gamma$ and $F'$ is a face in $\Gamma'$. 
Consider a homeomorphism $g:\partial F\to\partial F'$ which sends every vertex of $F$ to a vertex of $F'$, i.e. for $i\in\{1,2,3\}$ if $v_i$ are the vertices of $F$,  then $v'_i=g(v_i)$ are the vertices of $F'$. 
Such homeomorphisms will be called {\it special}. 

Consider a graph embedded in the connected sum $M\#M'$ whose vertex set is the union of the vertex sets of $\Gamma$ and $\Gamma'$, where each $v_i$ is identified with $v'_i$ and whose edge set is the union of the edge sets of $\Gamma$ and $\Gamma'$, where each edge $v_iv_j$ is identified with the edge $v'_iv'_j$. This embedded graph is called the {\it connected sum} of $\Gamma$ and $\Gamma'$ and it is denoted by $\Gamma\#_g\Gamma'$. 
Note that all faces of $\Gamma$ or $\Gamma'$ other than $F$ and $F'$ are faces of $\Gamma\#_g\Gamma'$. 
The connected sum can be depended on $g$. 

\begin{exmp}\label{ex3}{\rm 
Any connected sum of a $3$-gonal bipyramid and a tetrahedron is as in Fig. 3.
\begin{center}
\begin{tikzpicture}[scale=0.6]


\begin{scope}[rotate=30]
\draw[fill=black] (0,1.5) circle (2pt); 
\draw[fill=black] (0,-1) circle (2pt); 
\draw[fill=black] (-0.5,0) circle (2pt); 

\draw[fill=black] (-2,0.5) circle (2pt); 
\draw[fill=black] (2,0) circle (2pt); 
\draw[fill=black] (-2,-2) circle (2pt); 

\coordinate (A) at (0,1.5);
\coordinate (B) at (0,-1);
\coordinate (C) at (-0.5,0);
\coordinate (L) at (-2,0.5);
\coordinate (P) at (2,0);
\coordinate (LD) at (-2,-2);

\draw[thick, dashed, line width=1.75pt] (A)--(B);
\draw[thick, line width=1.75pt] (A)--(C); 
\draw[thick, line width=1.75pt] (B)--(C);

\draw[thick] (L)--(A);
\draw[thick, dashed] (L)--(B);
\draw[thick] (L)--(C);

\draw[thick] (P)--(A);
\draw[thick] (P)--(B);
\draw[thick] (P)--(C);

\draw[thick] (LD)--(B);
\draw[thick] (LD)--(C);
\draw[thick] (LD)--(L);

\end{scope}

\end{tikzpicture}
\captionof{figure}{ }
\end{center}
}\end{exmp}

\begin{exmp}\label{ex4}{\rm 

Any connected sum of two copies of a $3$-gonal bipyramid is one of three triangulations presented in Fig. 4. 
\begin{center}
\begin{tikzpicture}[scale=0.7]

\begin{scope}[xshift=-1cm]
\draw[fill=black] (0,1.5) circle (2pt); 
\draw[fill=black] (0,-1) circle (2pt); 
\draw[fill=black] (-0.5,0) circle (2pt); 

\draw[fill=black] (-2,0.5) circle (2pt); 
\draw[fill=black] (1.5,0.5) circle (2pt); 
\draw[fill=black] (-2,-1.5) circle (2pt); 
\draw[fill=black] (1.5,-1.5) circle (2pt); 

\coordinate (A) at (0,1.5);
\coordinate (B) at (0,-1);
\coordinate (C) at (-0.5,0);
\coordinate (L) at (-2,0.5);
\coordinate (P) at (1.5,0.5);
\coordinate (LD) at (-2,-1.5);
\coordinate (PD) at (1.5,-1.5);

\draw[thick, dashed, line width=1.75pt] (A)--(B);
\draw[thick, line width=1.75pt] (A)--(C); 
\draw[thick, line width=1.75pt] (B)--(C);

\draw[thick] (L)--(A);
\draw[thick, dashed] (L)--(B);
\draw[thick] (L)--(C);

\draw[thick] (P)--(A);
\draw[thick, dashed] (P)--(B);
\draw[thick] (P)--(C);

\draw[thick] (LD)--(B);
\draw[thick] (LD)--(C);
\draw[thick] (LD)--(L);

\draw[thick] (PD)--(B);
\draw[thick] (PD)--(C);
\draw[thick] (PD)--(P);

\node at (0,-2) {$(1)$};
\end{scope}
\begin{scope}[xshift=3.75cm, yshift=-0.25cm]
\draw[fill=black] (0,1.5) circle (2pt); 
\draw[fill=black] (0,-1) circle (2pt); 
\draw[fill=black] (0.5,0.1) circle (2pt); 

\draw[fill=black] (-1.5,0.5) circle (2pt); 
\draw[fill=black] (2,0.5) circle (2pt); 
\draw[fill=black] (-1.25,-1.25) circle (2pt); 
\draw[fill=black] (1.5,2) circle (2pt); 

\coordinate (A) at (0,1.5);
\coordinate (B) at (0,-1);
\coordinate (C) at (0.5,0.1);
\coordinate (L) at (-1.5,0.5);
\coordinate (P) at (2,0.5);
\coordinate (LD) at (-1.25,-1.25);
\coordinate (PG) at (1.5,2);

\draw[thick, dashed, line width=1.75pt] (A)--(B);
\draw[thick, line width=1.75pt] (A)--(C); 
\draw[thick, line width=1.75pt] (B)--(C);

\draw[thick] (L)--(A);
\draw[thick, dashed] (L)--(B);
\draw[thick] (L)--(C);

\draw[thick, dashed] (P)--(A);
\draw[thick] (P)--(B);
\draw[thick] (P)--(C);

\draw[thick] (LD)--(B);
\draw[thick] (LD)--(C);
\draw[thick] (LD)--(L);

\draw[thick] (PG)--(A);
\draw[thick] (PG)--(C);
\draw[thick] (PG)--(P);

\node at (0,-1.75) {$(2)$};
\end{scope}
\begin{scope}[xshift=9cm, yshift=-0.25cm]
\draw[fill=black] (0,1.5) circle (2pt); 
\draw[fill=black] (0,-1) circle (2pt); 
\draw[fill=black] (0.35,0.1) circle (2pt); 

\draw[fill=black] (-1.5,0.5) circle (2pt); 
\draw[fill=black] (1.75,0.25) circle (2pt); 
\draw[fill=black] (-1.25,-1.25) circle (2pt); 
\draw[fill=black] (1.3,1.5) circle (2pt); 

\coordinate (A) at (0,1.5);
\coordinate (B) at (0,-1);
\coordinate (C) at (0.35,0.1);
\coordinate (L) at (-1.5,0.5);
\coordinate (P) at (1.75,0.25);
\coordinate (LD) at (-1.25,-1.25);
\coordinate (PG) at (1.3,1.5);

\draw[thick, dashed, line width=1.75pt] (A)--(B);
\draw[thick, line width=1.75pt] (A)--(C); 
\draw[thick, line width=1.75pt] (B)--(C);

\draw[thick] (L)--(A);
\draw[thick, dashed] (L)--(B);
\draw[thick] (L)--(C);

\draw[thick] (P)--(A);
\draw[thick] (P)--(B);
\draw[thick] (P)--(C);

\draw[thick] (LD)--(B);
\draw[thick] (LD)--(C);
\draw[thick] (LD)--(L);

\draw[thick, dashed] (PG)--(B);
\draw[thick] (PG)--(A);
\draw[thick] (PG)--(P);

\node at (0,-1.75) {$(3)$};
\end{scope}

\end{tikzpicture}
\captionof{figure}{ }
\end{center}
Note that (2) and (3) are different embeddings of isomorphic graphs. 
}\end{exmp}

Let $\{T^{(i)}\}_{i=1}^n$ be a sequence of tetrahedra where
$$\{v^{(i)}_1, v^{(i)}_2, v^{(i)}_3, v^{(i)}_4\}\text{ and }\{F^{(i)}_1, F^{(i)}_2, F^{(i)}_3, F^{(i)}_4\}$$ are the set of vertices and the set of faces of $T^{(i)}$, respectively (as in Example \ref{ex1}, the face $F^{(i)}_j$ does not contain $v^{(i)}_j$). 
Now, for each $i=1,\dots,n-1$ we define a special homeomorphism 
$$g_i:\partial F^{(i)}_{p_i}\to\partial F^{(i+1)}_{q_i},$$ 
where $p_i,q_i\in\{1,2,3,4\}$, such that 
$q_i\neq p_{i+1}$ for $i=1,\dots,n-2.$
In other words, the image of $g_{i}$ and the domain of $g_{i+1}$ are the boundaries of different faces in the tetrahedron $T_{i+1}$ for $i=1,\dots,n-2$. 

Using the sequence of special homeomorphisms $\{g_i\}_{i=1}^{n-1}$, for $k=1,\dots,n$ we define a {\it combinatorial tetrahedral chain $\Theta_k$} (or a {\it combinatorial tetrahedral chain of length $k$}) recursively as follows
\begin{enumerate}
\item[$\bullet$] $\Theta_1=T^{(1)}$,
\item[$\bullet$] $\Theta_k=\Theta_{k-1}\#_{g_{k-1}}T^{(k)}$ for $k=2,\dots,n$.
\end{enumerate} 
It is clear that a combinatorial tetrahedral chain is a triangulation of $\mathbb{S}^2$.
For $k=1,2,3$ a combinatorial tetrahedral chain of length $k$ is unique: $\Theta_1$ is a tetrahedron (see Example \ref{ex1}), $\Theta_2$ is the $3$-gonal bipyramid (see Example \ref{ex2}), $\Theta_3$ is the connected sum of a $3$-gonal bipyramid and a tetrahedron (see Example \ref{ex3}). 
By Example \ref{ex4} there are precisely three possibilities for a combinatorial tetrahedral chain of length $4$ (the connected sum of two copies of $BP_3$). 

\section{Main result}

From this moment, combinatorial tetrahedral chains will be called simply {\it tetrahedral chains}. 
The definition of tetrahedral chains can be seen as a step-by-step construction of $\Theta_n$. 
We start with a tetrahedron where we choose a face and glue another tetrahedron to it. 
In the next steps we glue a new tetrahedron to one of three faces of the tetrahedron attached in the previous step. 
Suppose that in each step of this construction faces are chosen with equal probability, i.e. $\frac{1}{4}$ for the first step and $\frac{1}{3}$ in each of the remaining steps. 
We denote by $p^{(k)}_n$ the probability that the construction produces a tetrahedral chain with precisely $k$ zigzags up to reversing.


\begin{theorem}\label{th1}
Every tetrahedral chain contains at most $3$ zigzags up to reversing and
\begin{enumerate}
\item[(1)] $\displaystyle{\lim_{n \to \infty} p^{(1)}_n=\frac{8}{15}}$,
\item[(2)] $\displaystyle{\lim_{n \to \infty} p^{(2)}_n=\frac{2}{5}}$,
\item[(3)] $\displaystyle{\lim_{n \to \infty} p^{(3)}_n=\frac{1}{15}}$. 
\end{enumerate}
The above sequences of probabilities converge at an exponential rate. 
\end{theorem}

\begin{exmp}\label{ex5}{\rm 
Consider a tetrahedral chain of length $5$ presented in Fig. 5. 
It contains precisely two vertices of degree $3$: one of them is a vertex of the first tetrahedron used in the construction and the other is a vertex of the last tetrahedron. 
These vertices are denoted by $a$ and $b$, see Fig. 5. 
\begin{center}
\begin{tikzpicture}[scale=0.8]

\draw[fill=black] (0,1.5) circle (2pt); 
\draw[fill=black] (0,-1) circle (2pt); 
\draw[fill=black] (-0.5,0) circle (2pt); 

\draw[fill=black] (-2,0.5) circle (2pt); 
\draw[fill=black] (1.5,0.5) circle (2pt); 
\draw[fill=black] (-2,-1.5) circle (2pt); 
\draw[fill=black] (1.5,-1.6) circle (2pt); 

\draw[fill=black] (2.5,-0.5) circle (2pt); 

\coordinate (A) at (0,1.5);
\coordinate (B) at (0,-1);
\coordinate (C) at (-0.5,0);
\coordinate (L) at (-2,0.5);
\coordinate (P) at (1.5,0.5);
\coordinate (LD) at (-2,-1.5);
\coordinate (PD) at (1.5,-1.6);

\coordinate (X) at (2.5,-0.5);

\draw[thick, dashed] (A)--(B);
\draw[thick] (A)--(C); 
\draw[thick] (B)--(C);

\draw[thick] (L)--(A);
\draw[thick, dashed] (L)--(B);
\draw[thick] (L)--(C);

\draw[thick] (P)--(A);
\draw[thick, dashed] (P)--(B);
\draw[thick] (P)--(C);

\draw[thick] (LD)--(B);
\draw[thick] (LD)--(C);
\draw[thick] (LD)--(L);

\draw[thick] (PD)--(B);
\draw[thick] (PD)--(C);
\draw[thick] (PD)--(P);

\draw[thick, dashed] (X)--(B);
\draw[thick] (X)--(P);
\draw[thick] (X)--(PD);

\node at (-2.2,-1.7) {$a$};
\node at (2.75,-0.5) {$b$};

\end{tikzpicture}
\captionof{figure}{ }
\end{center}
If $a$ is a vertex of the first tetrahedron, then the first four tetrahedra forms a tetrahedral chain isomorphic to (1) from Example \ref{ex4}. 
But if $b$ is a vertex of the first tetrahedron, then the first four tetrahedra forms a tetrahedral chain isomorphic to (3) from Example \ref{ex4}. 
This means that this tetrahedral chain can be obtained as a result of the construction in two different ways. 
}\end{exmp}
Example \ref{ex5} shows that the probability $p^{(k)}_n$ is not the probability that a tetrahedral chain randomly chosen from the family of all tetrahedral chains of length $n$ has precisely $k$ zigzags up to reversing (we assume that the probability of choice is equal for all tetrahedral chains from this family). 
To find the latter probability we need a classification of tetrahedral chains which is an open problem (see \cite{BJ} for a partial classification of some proper tetrahedral chains). 

\section{$Z$-monodromy}

Consider a face $F$ in a triangulation $\Gamma$ and denote its vertices by $a,b,c$. Then the set of all oriented edges of $F$ is 
$$\Omega(F)=\{ab,bc,ca,ac,cb,ba\},$$
where $xy$ is the edge from $x$ to $y$ for $x,y\in\{a,b,c\}$. If $e=xy$, then we write $-e$ for the edge $yx$. 
Let 
$$D_F=(ab,bc,ca)(ac,cb,ba),$$
i.e. $D_F$ is a permutation on $\Omega(F)$ which transfers each oriented edge of $F$ to the next edge according to one of the two orientations on the boundary of this face. 
Now, we define the permutation $M_F$ on $\Omega(F)$ called the {\it $z$-monodromy} of $F$. 
For any $e\in\Omega(F)$ we take $e_0\in\Omega(F)$ such that $D_F(e_0)=e$. 
Since every zigzag is completely determined by any pair of consecutive edges, we can find the zigzag containing the sequence $e_0,e$. 
The first element of $\Omega(F)$ contained in this zigzag after $e$ is denoted by $M_F(e)$. 
Note that if $M_F(e)=e'$, then $M_F(-e')=-e$. 

A special homeomorphism $g:\partial F\to\partial F'$ induces a bijection between $\Omega(F)$ and $\Omega(F')$ which sends an oriented edge $xy$ to the oriented edge $g(x)g(y)$. 
We denote this bijection also by $g$. 

There are precisely $7$ possibilities for the $z$-monodromy $M_F$ and each of them is realized (see \cite[Theorem 4.4]{PT2}):
\begin{enumerate}
\item[{\rm (M1)}] $M_{F}$ is the identity,
\item[{\rm (M2)}] $M_{F}=D_{F}$,
\item[{\rm (M3)}] $M_{F}=(-e_{1},e_{2},e_{3})(-e_{3},-e_{2},e_{1})$,
\item[{\rm (M4)}] $M_{F}=(e_{1},-e_{2})(e_{2},-e_{1})$, where {\rm}$e_{3}$ and $-e_{3}$ are fixed points{\rm},
\item[{\rm (M5)}] $M_{F}=(D_{F})^{-1}$,
\item[{\rm (M6)}] $M_{F}=(-e_{1},e_{3},e_{2})(-e_{2},-e_{3},e_{1})$,
\item[{\rm (M7)}] $M_{F}=(e_{1},e_{2})(-e_{2},-e_{1})$, where {\rm}$e_{3}$ and $-e_{3}$ are fixed points{\rm}
\end{enumerate}
where $(e_{1},e_{2},e_{3})$ is one of the cycles in $D_{F}$. 

The triangulation $\Gamma$ is locally $z$-knotted for $F$ if and only if one of the cases (M1)--(M4) is realized. 
By \cite[Theorem 4.7]{PT2}, a triangulation is $z$-knotted if and only if $z$-monodromies of all faces are of types (M1)--(M4). 
If $M_F$ is (M6) or (M7), then $|\mathcal{Z}(F)|=4$ and if $M_F$ is (M5), then $|\mathcal{Z}(F)|=6$ (see \cite[Remark 4.9]{PT2}). 
By Example \ref{ex1}, for a tetrahedron $T$ and for each face $F$ in $T$ we have $|\mathcal{Z}(F)|=6$, which means that the $z$-monodromies of all faces in $T$ are of type (M5). 

\begin{exmp}\label{ex6}{\rm 
Consider $3$-gonal bipyramid from Example \ref{ex2}. 
Let $F$ be the face containing the vertices $a,1,2$. 
The set of all oriented edges of $F$ is 
$$\Omega(F)=\{e_1,e_2,e_3,-e_3,-e_2,-e_1\},$$ 
where $e_1=12, e_2=2a, e_3=a1$ and $(e_1,e_2,e_3)$ is one of $3$-cycles in $D_F$. 
Thus, the zigzag passes through oriented edges of $F$ as follows
$$\dots,e_3,e_1,\dots,-e_3,-e_2,\dots,e_1,e_2,\dots$$
Therefore, 
$$M_F=(-e_1,e_2,e_3)(-e_3,-e_2,e_1),$$
i.e. the $z$-monodromy of $F$ is of type (M3). 
Since $F$ can be transferred to any other face of $BP_3$ by an automorphism, $z$-monodromies of all faces are of type (M3). 
}\end{exmp}

The following simple lemma will be applied later to tetrahedral chains. 
\begin{lemma}\label{lem1}
Suppose that $F$ is a face in a triangulation $\Gamma$ such that
$\mathcal{Z}(\Gamma)=\mathcal{Z}(F)$
and suppose that $T$ is a tetrahedron with the set of faces $\{F', F_1, F_2, F_3\}$.
Then, for any special homeomorphism $g:\partial F\to\partial F'$, we have
$$\mathcal{Z}(\Gamma\#_g T)=\mathcal{Z}(F_i)\text{ for }i=1,2,3.$$
\end{lemma}

\begin{proof}
Let $e_{1},e_{2},e_{3}$ be the non-oriented edges of $F$. 
The connected sum $\Gamma\#_g T$ can be obtained from $\Gamma$ by adding a vertex in the interior of $F$ and three edges joining this vertex to the vertices of $F$. 
We denote a new edge by $e'_i$ if this edge and $e_i$ do not have a common vertex, see Fig. 6. 
\begin{center}
\begin{tikzpicture}[scale=1.2]

\draw[fill=black] (0,0) circle (1.5pt);
\draw[fill=black] (-60:2cm) circle (1.5pt);
\draw[fill=black] (-120:2cm) circle (1.5pt);
\draw[thick, line width=1pt] (-60:2cm) -- (-120:2cm);

\draw[thick, line width=1pt] (-60:2cm) -- (0,0) -- (-120:2cm);

\draw [thick, line width=1pt] (-60:2cm) -- (0,0);

\draw [thick, line width=1pt] (0,0) -- (-120:2cm);

\draw [thick, line width=1pt] (-120:2cm) -- (-60:2cm);

\draw[thick, line width=1pt, dashed] (-90:1.1547cm) -- (0,0);

\draw [thick, line width=1pt] (-90:1.1547cm) -- (-120:2cm);

\draw [thick, line width=1pt] (-90:1.1547cm) -- (-60:2cm);

\draw [thick, line width=1pt] (-90:1.1547cm) -- (0,0);

\draw[fill=black] (-90:1.1547cm) circle (1.5pt);

\node[xshift=2] at (-135:1cm) {$e_3$};
\node[xshift=-2] at (-47:1cm) {$e_2$};
\node at (0,-1.9cm) {$e_1$};

\node at (-75:1.2cm) {$e'_3$};
\node[xshift=-1.5] at (-105:1.17cm) {$e'_2$};
\node[xshift=1, yshift=1] at (-80:0.8cm) {$e'_1$};

\end{tikzpicture}
\captionof{figure}{ }
\end{center}
So, $F$ is replaced by the faces $F_1, F_2, F_3$. 
Since all zigzags of $\Gamma$ contain edges of $F$, then each of zigzags of $\Gamma\#_g T$ passes through at least one of edges $e_1,e_2,e_3$ (otherwise, there is a zigzag containing only edges $e'_1, e'_2, e'_3$, which is impossible). 
Without loss of generality we assume that $Z\in\mathcal{Z}(\Gamma\#_g T)$ contains $e_1$. 
Then, $Z$ is 
$$\dots, e_1,e'_3,e'_1,e_3,\dots\text{ or }\dots,e_1,e'_2,e'_1,e_2,\dots$$
or a zigzag reversed to one of them. 
Therefore, $Z$ passes through edges of all faces $F_1, F_2, F_3$. 
\end{proof}

It was pointed in Example \ref{ex1} that each of zigzags of $\Theta_1=T$ passes through edges of all its faces. 
Using Lemma \ref{lem1}, we establish a similar property for tetrahedral chains of length $n\geq 2$. 

\begin{lemma}\label{lem2}
Let $n\geq 2$ and $\{\Theta_k\}_{k=1}^{n-1}$ be a sequence of tetrahedral chains, i.e. 
$\Theta_k$ is obtained by gluing a tetrahedron to a face of $\Theta_{k-1}$ that comes from a tetrahedron glued in the previous step. 
Let also $\{F',F_1,F_2,F_3\}$ be the set of faces of $T^{(n)}$ and let $g_{n-1}:\partial F\to F'$ be any special homeomorphism such that $F$ is any face of $\Theta_{n-1}$ that comes from $T^{(n-1)}$. 
If $\Theta_n=\Theta_{n-1}\#_{g_{n-1}}T^{(n)}$, then
$$\mathcal{Z}(\Theta_{n})=\mathcal{Z}(F_i)\text{ for }i=1,2,3.$$
\end{lemma}

\begin{proof}
We prove the lemma by induction. 
The tetrahedral chain $\Theta_2=BP_3$ is $z$-knotted (see Example \ref{ex2}), so all its zigzags passes through edges of all faces. 
Suppose that $\Theta_{n-1}$ was obtained by gluing $\Theta_{n-2}$ and $T^{(n-1)}$ together and let $\mathcal{F}$ be the set consisting of three faces of $\Theta_{n-1}$ that come from $T^{(n-1)}$. 
Assume that $\mathcal{Z}(\Theta_{n-1})=\mathcal{Z}(F)$ for all $F\in\mathcal{F}$. 
If we set $F\in\mathcal{F}$, then for any special homeomorphism ${g_{n-1}:\partial F\to\partial F'}$ we obtain 
$\Theta_n=\Theta_{n-1}\#_{g_{n-1}}T^{(n)}$ ,
where $F$ is replaced by $F_1,F_2,F_3$. 
By Lemma \ref{lem1}, $\mathcal{Z}(\Theta_{n})=\mathcal{Z}(F_i)$ for all $i=1,2,3$. 
\end{proof}

Example \ref{ex1} and Lemma \ref{lem2} shows that any tetrahedral chain contains at most $3$ zigzags up to reversing. 
Lemma \ref{lem2} together with the properties $z$-monodromies presented at the beginning of this section imply the following. 

\begin{prop}\label{prop1}
Let $\Theta_n$ and $\Theta_{n-1}$ be as in Lemma \ref{lem2}. 
For every face $F$ of $\Theta_n$ that comes from $T^{(n)}$ the following assertions are fulfilled: 
\begin{enumerate}
\item[(1)] if $M_F$ is of type {\rm (M1)}--{\rm (M4)}, then $\Theta_n$ is $z$-knotted, 
\item[(2)] if $M_F$ is of type {\rm (M6)} or {\rm (M7)}, then $\Theta_n$ contains precisely $2$ zigzags up to reversing, 
\item[(3)] if $M_F$ is of type {\rm (M5)}, then $\Theta_n$ contains precisely $3$ zigzags up to reversing.
\end{enumerate}
\end{prop}

\section{The digraph of $z$-monodromies}

Consider a triangulation $\Gamma$ and a face $F$ in this triangulation. 
We triangulate $F$ as in the proof of Lemma \ref{lem1}, i.e. we add a vertex in its interior and three edges joining this vertex to the vertices of $F$. 
We obtain a new triangulation $\Gamma'$. 
Since the structure of $\Gamma'$ outside the three new faces is not changed, the $z$-monodromies of these faces depend only on the type of $M_F$ and do not depend on the choice of $\Gamma$. 
If $M_F$ is of type (M$i$), then we denote by $\mathcal{M}_i$ the set of types of $z$-monodromies of faces obtained from $F$ by the above operation. 

Using this observation, we define a digraph G whose vertex set consists of all seven types of $z$-monodromies (M1)--(M7). 
The digraph G contains a directed edge from (M$i$) to (M$j$) if ${\rm (M}j{)}\in\mathcal{M}_i$. 
Note that G has loops if there exists $i\in\{1,\dots,7\}$ such that ${\rm (M}i{)}\in\mathcal{M}_i$. 

\begin{lemma}
The following assertions are fulfilled:
\begin{enumerate}
\item $\mathcal{M}_1=\{{\rm (M4)}\}$, 
\item $\mathcal{M}_2=\{{\rm (M5)}\}$,
\item $\mathcal{M}_3=\{{\rm (M6)}, {\rm (M7)}\}$ (the $z$-monodromies of two faces are of type ${\rm (M7)}$ and the $z$-monodromy of the remaining face is of type ${\rm (M6)}$),
\item $\mathcal{M}_4=\{{\rm (M1)}, {\rm (M3)}\}$ (the $z$-monodromies of two faces are of type ${\rm (M3)}$ and the $z$-monodromy of the remaining face is of type ${\rm (M1)}$),
\item $\mathcal{M}_5=\{{\rm (M3)}\}$,
\item $\mathcal{M}_6=\{{\rm (M2)}, {\rm (M4)}\}$ (the $z$-monodromies of two faces are of type ${\rm (M4)}$ and the $z$-monodromy of the remaining face is of type ${\rm (M2)}$),
\item $\mathcal{M}_7=\{{\rm (M6)}, {\rm (M7)}\}$ (the $z$-monodromies of two faces are of type ${\rm (M6)}$ and the $z$-monodromy of the remaining face is of type ${\rm (M1)}$).
\end{enumerate}
\end{lemma}
Therefore, G is the digraph presented in Fig. 7:
\begin{center}
\begin{tikzpicture}[scale=2]
\begin{scope}[every node/.style={circle,thick,draw}]
    \node (M1) at (3.5,0) {M1};
    \node (M2) at (-1,-1) {M2};
    \node (M3) at (1,1) {M3};
    \node (M4) at (2,0) {M4};
    \node (M5) at (-1,1) {M5};
    \node (M6) at (1,-1) {M6} ;
    \node (M7) at (0,0) {M7} ;
\end{scope}

\begin{scope}[{>=stealth},
              every node/.style={fill=white,circle, minimum size = 0.1cm, inner sep=2pt},
              every edge/.style={draw=black,very thick}]
    \path [->] (M1) edge [bend right] node {$1$} (M4);
    \path [->] (M2) edge node {$1$} (M5);
    \path [->] (M3) edge node {$\frac{1}{3}$} (M6);
    \path [->] (M3) edge node {$\frac{2}{3}$} (M7);
    \path [->] (M4) edge [bend right] node {$\frac{1}{3}$} (M1);
    \path [->] (M4) edge node {$\frac{2}{3}$} (M3);
    \path [->] (M5) edge node {$1$} (M3);
    \path [->] (M6) edge node {$\frac{1}{3}$} (M2);
    \path [->] (M6) edge node {$\frac{2}{3}$} (M4);
    \path [->] (M7) edge node {$\frac{2}{3}$} (M6);
    \path [->] (M7) edge[out=135,in=225,looseness=8] node {$\frac{1}{3}$} (M7); 
\end{scope}
\end{tikzpicture}
\captionof{figure}{ }
\end{center}
We assign labels to edges of G according to the following rule: 
if $m$ of three faces obtained by triangulating a face with the $z$-monodromy (M$i$) are of type (M$j$), then the edge from (M$i$) to (M$j$) is labelled by $p_{ij}=\frac{m}{3}$ (see Fig. 7). 
For tetrahedral chains, $p_{ij}$ is the probability that after gluing a tetrahedron to a face of $\Theta_{n-1}$ with the $z$-monodromy (M$i$) (where $n\geq 2$) we choose a face in $\Theta_n$ that comes from the tetrahedron such that its $z$-monodromy is (M$j$). 
This face will be used in the next step to construct $\Theta_{n+1}$.

\begin{proof}
Let $\Omega(F)=\{e_1,e_2,e_3,-e_3,-e_2,-e_1\},$ where $(e_1,e_2,e_3)$ is one of $3$-cycles in $D_F$. 
Every oriented edge added after triangulating $F$ is denoted by $e'_i$ if it does not have a common vertex with $e_i$ and it is oriented from a vertex coming from $F$ to the new vertex (see Fig. 8). 
If we change its orientation, then we get the edge $-e'_i$. 
The face obtained by triangulating $F$ whose one of oriented edges is $e_i$ will be denoted by $F_i$. 
\begin{center}
\begin{tikzpicture}[scale=1.4]

\begin{scope}
\draw[fill=black] (0,0) circle (1.5pt);
\draw[fill=black] (-60:2cm) circle (1.5pt);
\draw[fill=black] (-120:2cm) circle (1.5pt);
\draw[thick, line width=1pt] (-60:2cm) -- (-120:2cm);

\draw[thick, line width=1pt] (-60:2cm) -- (0,0) -- (-120:2cm);

\draw [thick, dashed, decoration={markings,
mark=at position 0.55 with {\arrow[scale=2,>=stealth]{<}}},
postaction={decorate}] (0,0) -- (-60:2cm);

\draw [thick, dashed, decoration={markings,
mark=at position 0.55 with {\arrow[scale=2,>=stealth]{<}}},
postaction={decorate}] (-120:2cm) -- (0,0);

\draw [thick, dashed, decoration={markings,
mark=at position 0.55 with {\arrow[scale=2,>=stealth]{<}}},
postaction={decorate}] (-60:2cm) -- (-120:2cm);

\node[xshift=2] at (-135:1.2cm) {$e_3$};
\node[xshift=-2] at (-47:1.2cm) {$e_2$};
\node at (0,-2cm) {$e_1$};

\end{scope}

\begin{scope}[xshift=4.5cm]
\draw[fill=black] (0,0) circle (1.5pt);
\draw[fill=black] (-60:2cm) circle (1.5pt);
\draw[fill=black] (-120:2cm) circle (1.5pt);
\draw[thick, line width=1pt] (-60:2cm) -- (-120:2cm);

\draw[thick, line width=1pt] (-60:2cm) -- (0,0) -- (-120:2cm);

\draw [thick, dashed, decoration={markings,
mark=at position 0.55 with {\arrow[scale=2,>=stealth]{<}}},
postaction={decorate}] (0,0) -- (-60:2cm);

\draw [thick, dashed, decoration={markings,
mark=at position 0.55 with {\arrow[scale=2,>=stealth]{<}}},
postaction={decorate}] (-120:2cm) -- (0,0);

\draw [thick, dashed, decoration={markings,
mark=at position 0.55 with {\arrow[scale=2,>=stealth]{<}}},
postaction={decorate}] (-60:2cm) -- (-120:2cm);






\draw [thick, dashed, decoration={markings,
mark=at position 0.55 with {\arrow[scale=2,>=stealth]{<}}},
postaction={decorate}] (-90:1.1547cm) -- (-120:2cm);

\draw [thick, dashed, decoration={markings,
mark=at position 0.55 with {\arrow[scale=2,>=stealth]{<}}},
postaction={decorate}] (-90:1.1547cm) -- (-60:2cm);

\draw [thick, dashed, decoration={markings,
mark=at position 0.55 with {\arrow[scale=2,>=stealth]{<}}},
postaction={decorate}] (-90:1.1547cm) -- (0,0);

\draw [thick, line width=1pt] (-90:1.1547cm) -- (-120:2cm);

\draw [thick, line width=1pt] (-90:1.1547cm) -- (-60:2cm);

\draw [thick, line width=1pt] (-90:1.1547cm) -- (0,0);

\draw[fill=black] (-90:1.1547cm) circle (1.5pt);

\node[xshift=2] at (-135:1.2cm) {$e_3$};
\node[xshift=-2] at (-47:1.2cm) {$e_2$};
\node at (0,-2cm) {$e_1$};

\node[xshift=0.15cm, yshift=-0.05cm] at (-75:1.2cm) {$e'_3$};
\node[xshift=-0.15cm, yshift=-0.05cm] at (-105:1.17cm) {$e'_2$};
\node[xshift=1, yshift=1] at (-80:0.8cm) {$e'_1$};
\end{scope}

\draw [xshift=1.75cm, ->]  (-0.25cm,-0.75) -- (1.25cm,-0.75);

\end{tikzpicture}
\captionof{figure}{ }
\end{center}

(1). Suppose that the $z$-monodromy of $F$ is of type (M1), i.e. $M_F$ is the identity. 
Thus, there is a single zigzag (up to reversing) which passes through edges of faces $F_i$ as follows:
$$\dots, e_1, e'_3,-e'_1,e_3,\dots,e_3,e'_2,-e'_3,e_2,\dots, e_2,e'_1,-e'_2,e_1,\dots$$
Note that for each $i=1,2,3$ this zigzag passes through the edges of $F_i$ and it can be written as
$$\dots, e_i, D_{F_i}(e_i),\dots, -(D_{F_i})^{-1}(e_i),-D_{F_i}(e_i),\dots,(D_{F_i})^{-1}(e_i),e_i,\dots$$
Thus, it is sufficient to find the type of $M_{F_1}$ and the $z$-monodromies of $F_2$ and $F_3$ are of the same type. 
For $F_1$ the zigzag is
$$\dots, e_1,e'_3,\dots, e'_2,-e'_3,\dots,-e'_2,e_1,\dots,$$
so $M_{F_1}=(e'_3,e'_2)(-e'_2,-e'_3)$. 
If we rename oriented edges of $F_1$ as below
$$e_1=E_3, e'_3=E_1, -e'_2=E_2,$$
then $(E_1,E_2,E_3)$ is one of the cycles in $D_{F_1}$ and $M_{F_1}=(E_1,-E_2)(E_2,-E_1)$. 
Therefore, $M_{F_i}$ is of type (M4) for $i=1,2,3$ and $\mathcal{M}_1=\{{\rm (M4)}\}$. 

(2). Let the $z$-monodromy of $F$ be of type (M2), i.e. $M_F=D_F$. 
There are precisely three zigzags (up to reversing) passing through edges of faces $F_i$: 
$$\dots,e_1,e'_3,-e'_1,e_3,\dots\text{ and }\dots,e_2,e'_1,-e'_2,e_1,\dots\text{ and }\dots,e_3,e'_2,-e'_3,e_2,\dots$$
Observe that each of the faces $F_1,F_2,F_3$ occurs in each of these zigzags. 
In other words, $|\mathcal{Z}(F_i)|=6$ and $M_{F_i}=(D_{F_i})^{-1}$ for every $i=1,2,3$. 
We establish that $\mathcal{M}_2=\{{\rm (M5)}\}$. 

(3). Assume that the $z$-monodromy of $F$ is of type (M3). 
If 
$$M_{F}=(-e_{1},e_{2},e_{3})(-e_{3},-e_{2},e_{1}),$$ 
then there are precisely two zigzags (up to reversing) which pass through edges of faces $F_i$:
$$\dots, e_1,e'_3,-e'_1,e_3,\dots,-e_1,e'_2,-e'_1,-e_2,\dots\text{ and }\dots,-e_2,e'_3,-e'_2,-e_3,\dots$$
For $F_1$ these sequences reduce to 
$$\dots,e_1,e'_3,\dots,-e_1,e'_2,\dots\text{ and }\dots,e'_3,-e'_2,\dots,$$
so $M_{F_1}=(-e_1,-e'_2,e'_3)(-e'_3,e'_2,e_1)$. 
By renaming the oriented edges of $F_1$ as below
$$e_1=E_1, e'_3=E_2, -e'_2=E_3,$$
we get $M_{F_1}=(-E_1,E_3,E_2)(-E_2,-E_3,E_1)$ where $(E_1,E_2,E_3)$ is one of the cycles in $D_{F_1}$, i.e. $M_{F_1}$ is of type (M6). 
For $i=2,3$ the zigzags pass through edges of $F_i$ as follows:
$$\dots,(D_{F_i})^{-1}(x),x,\dots,D_{F_i}(x),(D_{F_i})^{-1}(x),\dots\text{ and }\dots,x,D_{F_i}(x),\dots,$$
where $x=-e_2$ if $i=2$ and $x=e_3$ if $i=3$. 
Thus, $M_{F_2}$ and $M_{F_3}$ are of the same type and it is sufficient to find the $z$-monodromy of one of faces $F_2, F_3$. 
The zigzags passing through edges of $F_2$ are
$$\dots,e'_3,-e'_1\dots,-e'_1,-e_2\dots\text{ and }\dots,-e_2,e'_3,\dots$$
and $M_{F_2}=(-e'_3,e_2)(-e_2,e'_3)$. 
If
$$e_2=E_2,e'_1=E_3,-e'_3=E_1,$$
then $(E_1,E_2,E_3)$ is one of the cycles in $D_{F_2}$ and $M_{F_2}=(E_1,E_2)(-E_2,-E_1)$, i.e. the $z$-monodromies of $F_2$ and $F_3$ are of type (M7). 
Therefore, $\mathcal{M}_3=\{{\rm (M6)}, {\rm (M7)}\}$. 

(4). Suppose that the $z$-monodromy of $F$ is of type (M4). 
If 
$$M_{F}=(e_{1},-e_{2})(e_{2},-e_{1}),$$ 
then there is a single zigzag (up to reversing) passing through edges of faces $F_i$: 
$$\dots,e_1,e'_3,-e'_1,e_3,\dots,e_3,e'_2,-e'_3,e_2,\dots,-e_1,e'_2,-e'_1,-e_2,\dots$$
For $i=1,2$ the zigzag passes through edges of $F_i$ as follows: 
$$\dots,x,D_{F_i}(x),\dots,-(D_{F_i})^{-1}(x),-D_{F_i}(x),\dots,-x,-(D_{F_i})^{-1}(x),\dots,$$
where $x=e_1$ if $i=1$ and $x=-e'_3$ if $i=2$. 
Thus, the $z$-monodromies of $F_1$ and $F_2$ are of the same type. 
The zigzag passing through edges of $F_1$ is
$$\dots,e_1,e'_3,\dots,e'_2,-e'_3,\dots,-e_1,e'_2,\dots$$
and $M_{F_1}=(e'_2,e_1,e'_3)(-e'_3,-e_1,-e'_2)$. 
As previous, we change the notation
$$e_1=E_2, e'_3=E_3, -e'_2=E_1$$
and we get $M_{F_1}=(-E_1,E_2,E_3)(-E_3,-E_2,E_1)$ where $(E_1,E_2,E_3)$ is one of the cycles in $D_{F_1}$. 
So, $M_{F_1}$ and $M_{F_2}$ are of type (M3). 
Now, consider $F_3$. The zigzag passes through edges of this face as follows
$$\dots,-e'_1,e_3,\dots,e_3,e'_2,\dots,e'_2,-e'_1,\dots,$$
so $M_{F_3}$ is the identity (the type (M1)). 
Therefore $\mathcal{M}_4=\{{\rm (M1)}, {\rm (M3)}\}$. 

(5). Let the $z$-monodromy of $F$ be of type (M5), i.e. $M_F=(D_F)^{-1}$. 
Then there is a single zigzag (up to reversing) passing through edges of faces $F_i$:
$$\dots,e_1,e'_3,-e'_1,e_3,\dots,e_2,e'_1,-e'_2,e_1,\dots,e_3,e'_2,-e'_3,e_2,\dots$$
Note that for each $i=1,2,3$ the zigzag reduces to
$$\dots,e_i,D_{F_i}(e_i),\dots,(D_{F_i})^{-1}(e_i),e_i,\dots,-(D_{F_i})^{-1}(e_i),-D_{F_i}(e_i),\dots$$
and the $z$-monodromies of faces $F_1,F_2,F_3$ are of the same type. 
The edges of $F_1$ occur in this zigzag as follows
$$\dots,e_1,e'_3,\dots,-e'_2,e_1,\dots,e'_2,-e'_3,\dots$$
and $M_{F_1}=(-e_1,e'_3,-e'_2)(e'_2,-e'_3,e_1)$. 
We rename the oriented edges of $F_1$ as below
$$e_1=E_1,e'_3=E_2,-e'_2=E_3$$
and we get $M_{F_1}=(-E_1,E_2,E_3)(-E_3,-E_2,E_1)$ where $(E_1,E_2,E_3)$ is one of the cycles in $D_{F_1}$. 
Thus $M_{F_i}$ is of type (M3) for $i=1,2,3$ and $\mathcal{M}_5=\{{\rm (M3)}\}$. 

(6). Assume that the $z$-monodromy of $F$ is of type (M6). 
If 
$$M_{F}=(-e_{1},e_{3},e_{2})(-e_{2},-e_{3},e_{1}),$$ 
then there is a single zigzag (up to reversing) which passes through edges of faces $F_i$: 
$$\dots,e_1,e'_3,-e'_1,e_3,\dots,e_2,e'_1,-e'_2,e_1,\dots,-e_2,e'_3,-e'_2,-e_3,\dots$$ 
For $F_1$ this zigzag reduces to 
$$\dots,e_1,e'_3,\dots,-e'_2,e_1,\dots,e'_3,-e'_2,\dots$$
and we establish that the $z$-monodromy of $F_1$ is $M_{F_1}=(e_1,e'_3,-e'_2)(e'_2,-e'_3,-e_1)$. 
If 
$$e_1=E_1,e'_3=E_2,-e'_2=E_3,$$
then $(E_1,E_2,E_3)$ is one of the cycles in $D_{F_1}$ and $M_{F_1}=(E_1,E_2,E_3)(-E_3,-E_2,-E_1)=D_{F_1}$ (the type (M2)). 
For $i=2,3$ the zigzag is a sequence of form: 
$$\dots,x,D_{F_i}(x),\dots,-x,-(D_{F_i})^{-1}(x),\dots,-(D_{F_i})^{-1}(x),-D_{F_i}(x),\dots,$$
where $x=e_2$ if $i=2$ and $x=-e'_1$ if $i=3$. 
In other words, the $z$-monodromies of $F_2$ and $F_3$ are of the same type. 
The zigzag passes through edges of $F_2$ as follows
$$\dots,e'_3,-e'_1,\dots,e_2,e'_1,\dots,-e_2,e'_3,\dots$$
and $M_{F_2}=(e_2,-e'_1)(e'_1,-e_2)$. 
If we change the notation
$$e_2=E_1, e'_1=E_2, -e'_3=E_3,$$
then $(E_1,E_2,E_3)$ is one of the cycles in $D_{F_2}$ and $M_{F_2}=(E_1,-E_2)(E_2,-E_1)$. 
The $z$-monodromies of $F_2$ and $F_3$ are of type (M4) and $\mathcal{M}_6=\{{\rm (M2)}, {\rm (M4)}\}$. 

(7). Suppose that the $z$-monodromy of $F$ is of type (M7). 
If 
$$M_{F}=(e_{1},e_{2})(-e_{2},-e_{1}),$$ 
then there are precisely two zigzags (up to reversing) which pass through edges of faces $F_i$: 
$$\dots, e_1,e'_3,-e'_1,e_3,\dots,e_3,e'_2,-e'_3,e_2,\dots\text{ and }\dots,e_2,e'_1,-e'_2,e_1,\dots$$
For $i=1,2$ the zigzags are sequences of form
$$\dots,x,D_{F_i}(x),\dots,-(D_{F_i})^{-1}(x),-D_{F_i}(x),\dots\text{ and }\dots,(D_{F_i})^{-1}(x),x,\dots,$$ 
where $x=e_1$ if $i=1$ and $x=e'_1$ if $i=2$, i.e. the $z$-monodromies of these faces are of the same type. 
For $F_1$ the zigzags are sequences
$$\dots,e_1,e'_3,\dots,e'_2,-e'_3,\dots\text{ and }\dots,-e'_2,e_1,\dots,$$
thus, $M_{F_1}=(-e'_3,e_1,-e'_2)(e'_2,-e_1,e'_3)$. 
If 
$$e_1=E_3, e'_3=E_1, -e'_2=E_2,$$
where $(E_1,E_2,E_3)$ is one of the cycles in $D_{F_1}$, then $M_{F_1}=(-E_1,E_3,E_2)(-E_2,-E_3,E_1)$. 
So, $M_{F_1}$ and $M_{F_2}$ are of type (M6). 
The zigzags pass through edges of $F_3$ as follows 
$$\dots,-e'_1,e_3,\dots,e_3,e'_2,\dots\text{ and }\dots,e'_1,-e'_2,\dots$$ 
and $M_{F_3}=(e'_2,-e'_1)(e'_1,-e'_2)$. 
We rename oriented edges of $F_3$ as follows
$$e_3=E_3, e'_2=E_1, -e'_1=E_2$$
and $M_{F_3}=(E_1,E_2)(-E_2,-E_1)$, where $(E_1,E_2,E_3)$ is one of the cycles in $D_{F_3}$. 
This $z$-monodromy is of type (M7) and 
$\mathcal{M}_7=\{{\rm (M6)}, {\rm (M7)}\}$. 
\end{proof}

\section{Proof of Theorem \ref{th1}}
Our proof is based on well-known properties of Markov chains, see, for example \cite{HagMarkovCh}. 

Consider the time-homogeneous Markov chain $\{X_n\}_{n\in\mathbb{N}}$, where $X_n=j$ if and only if a face with the $z$-monodromy (M$j$) is chosen in $\Theta_{n}$ in the construction. 
Indeed, by Section 6, the probability 
of transition from $j_{n-1}$ to $j_n$ in the $n$-th step depends only on the state $j_{n-1}$ attained in the $(n-1)$-st step and this probability does not depend on $n\in\mathbb{N}$, i.e. 
if $\mathcal{S}=\{1,\dots,7\}$ is the state space, then for every $n\in\mathbb{N}$
$${\rm P}(X_n=j_n|X_{n-1}=j_{n-1},\dots,X_1=j_1)={\rm P}(X_n=j_n|X_{n-1}=j_{n-1})$$
for all $j_1,\dots,j_n\in\mathcal{S}$ and
$${\rm P}(X_n=j|X_{n-1}=i)={\rm P}(X_2=j|X_1=i)=p_{ij}$$
for all $i,j\in\mathcal{S}$.
Thus, the digraph G presented in Fig. 7 is the transition graph of $\{X_n\}_{n\in\mathbb{N}}$ and the matrix 
\[
P=[p_{ij}]=\begin{bmatrix}
    0 & 0 & 0 & 1 & 0 & 0 & 0 \\
    0 & 0 & 0 & 0 & 1 & 0 & 0 \\
    0 & 0 & 0 & 0 & 0 & \frac{1}{3} & \frac{2}{3} \\
    \frac{1}{3} & 0 & \frac{2}{3} & 0 & 0 & 0 & 0 \\
    0 & 0 & 1 & 0 & 0 & 0 & 0 \\
    0 & \frac{1}{3} & 0 & \frac{2}{3} & 0 & 0 & 0 \\
    0 & 0 & 0 & 0 & 0 & \frac{2}{3} & \frac{1}{3} 
\end{bmatrix}
\]
is the transition matrix of this Markov chain. 
It is easy to see that any two states $i,j\in\mathcal{S}$ communicate and $\{X_n\}_{n\in\mathbb{N}}$ is irreducible. 
There is a loop at the state $7$, thus this state is aperiodic and $\{X_n\}_{n\in\mathbb{N}}$ is also aperiodic. 

Let $p_{ij}(n)=(P^n)_{ij}$ be the $n$-step transition probability in $\{X_n\}_{n\in\mathbb{N}}$. 
Since $\{X_n\}_{n\in\mathbb{N}}$ is irreducible and aperiodic finite Markov chain, then it is ergodic with a unique stationary distribution $\pi=[\pi_1,\dots,\pi_7]$ such that
$$\displaystyle{\lim_{n \to \infty} p_{ij}(n)=\pi_j}$$
for all states $i,j\in\mathcal{S}$.
A direct verification shows that the stationary distribution is 
$$\pi=\begin{bmatrix}
    \frac{1}{15} & \frac{1}{15} & \frac{1}{5} & \frac{1}{5} & \frac{1}{15} & \frac{1}{5} & \frac{1}{5} 
\end{bmatrix}\!\!.$$

Recall that $p^{(k)}_n$ is the probability that the construction produces a tetrahedral chain $\Theta_n$ with precisely $k$ zigzags up to reversing. 
Proposition \ref{prop1} implies that
$$\displaystyle{\lim_{n \to \infty} p^{(1)}_n=\lim_{n \to \infty}(p_{i1}(n)+p_{i2}(n)+p_{i3}(n)+p_{i4}(n))=\pi_1+\pi_2+\pi_3+\pi_4=\frac{8}{15}},$$
$$\displaystyle{\lim_{n \to \infty} p^{(2)}_n=\lim_{n \to \infty}(p_{i6}(n)+p_{i7}(n))=\pi_6+\pi_7=\frac{2}{5}},$$
$$\displaystyle{\lim_{n \to \infty} p^{(3)}_n=\lim_{n \to \infty}p_{i5}(n)=\pi_5=\frac{1}{15}}$$
for all $i\in\mathcal{S}$. 

Recall that an irreducible and aperiodic Markov chain with finite state space converges exponentially, i.e.  
for our Markov chain $\{X_n\}_{n\in\mathbb{N}}$, there exist constants $c>0$ and $\gamma\in(0,1)$ such that
$$|p_{ij}(n)-\pi_j|\leq c\gamma^n$$
for all $i,j\in\mathcal{S}$. 
Let $I_1=\{1,2,3,4\}, I_2=\{6,7\}, I_3=\{5\}$ and let denote $\displaystyle{L_k=\lim_{n \to \infty} p^{(k)}_n}$ for all $k\in\{1,2,3\}$. 
We have
$$|p^{(k)}_n-L_k|=\displaystyle{|\sum_{j\in I_k} ({p_{ij}(n)-\pi_j})|}\leq\displaystyle{\sum_{j\in I_k} |p_{ij}(n)-\pi_j|}\leq|I_k|c\gamma^n=2^{3-k}c\gamma^n\leq c'\gamma^n,$$
where $c'=4c$. 
Thus, $p^{(k)}_n$ converges at an exponential rate.


\begin{thebibliography}{99}

\bibitem{BJ} Babiker H., Janeczko S., {\it Combinatorial representation of tetrahedral chains}, 
Commun. Inf. Syst. 15 (2015), 331-359.

\bibitem{BD}
Brinkmann G., Dress, A. W. M.,
{\it PentHex puzzles. A reliable and efficient top-down approach to fullerene-structure enumeration},
Adv. Appl. Math. 21 (1998), 473--480.

\bibitem{Coxeter} Coxeter H.S.M., 
{\it Regular polytopes}, 
Dover Publications, New York 1973 (3rd ed).

\bibitem{CrRos}
Crapo H., Rosenstiehl P., {\it On lacets and their manifolds}, Discrete Math. 233 (2001), 299--320.

\bibitem{DDS-book} Deza M., Dutour Sikiri\'c M., Shtogrin M.,
{\it Geometric Structure of Chemistry-relevant Graphs: zigzags and central circuit}, 
Springer 2015.

\bibitem{ElWa}
Elgersma M., Wagon S., {\it An Asymptotically Closed Loop of Tetrahedra}, 
Math. Intell. 39, 40–45 (2017).

\bibitem{HagMarkovCh}
Häggström O., {\it Finite Markov Chains and Algorithmic Applications}, 
London Mathematical Society Student Texts, Cambridge University Press (2002).

\bibitem{GR-book}
Godsil C., Royle G., {\it Algebraic Graph Theory}, 
Graduate Texts in Mathematics 207, Springer 2001. 

\bibitem{TriApp}
Hjelle Ø., Dæhlen M., {\it Triangulations and Applications}, Springer 2006. 

\bibitem{Lins2} 
Lins S., Oliveira-Lima E., Silva V., {\it A homological solution for the Gauss code problem in arbitrary surfaces}, 
J. Combin. Theory, Ser. B 98 (2008), 506--515.

\bibitem{MT-book}
Mohar B., Thomassen C., {\it Graphs on Surfaces}, The Johns Hopkins University Press 2001.

\bibitem{PT1} 
Pankov M., Tyc A., {\it Connected sums of z-knotted triangulations}, 
Euro. J. Comb. 80 (2019), 326--338.

\bibitem{PT3} 
Pankov M., Tyc A., {\it On two types of $z$-monodromy in triangulations of surfaces},
Discrete Math. 342 (2019), 2549--2558.

\bibitem{PT2}
Pankov M., Tyc A., {\it $z$-Knotted Triangulations of Surfaces}, Discrete Comput. Geom. 66 (2021), 636--658. 

\bibitem{Shank} Shank H., 
{\it The theory of left-right paths} in Combinatorial Mathematics III,
Lecture Notes in Mathematics 452, Springer 1975, 42--54.

\bibitem{HS-prob}
Steinhaus H., {\it Probl\`eme 175}, 
Colloq. Math. 4 (1957), 243.

\bibitem{Stewart}
Stewart I., {\it Tetrahedral chains and a curious semigroup}, 
Extracta Math. 34 (2019), 99--122.

\bibitem{SSw}
Świerczkowski S., {\it On chains of regular tetrahedra}, 
Colloq. Math. 7 (1959), 9--10.

\bibitem{T2}
Tyc A., {\it  $Z$-knotted and z-homogeneous triangulations of surfaces}, 
Discrete Math. 344 (2021), 112405.

\bibitem{T1}
Tyc A., {\it $Z$-oriented triangulations of surfaces}, 
Ars Math. Contemp. 22 (2022), p. \#1.02.

\end{thebibliography}
\end{document}